\documentclass[11pt]{amsart}

\usepackage{mathrsfs,amsfonts,amsmath,amssymb}
\usepackage{latexsym}
\newtheorem{satz}{Theorem}

\newtheorem{theorem}[satz]{Theorem}
\newtheorem{lemma}[satz]{Lemma}

\newtheorem{corollary}[satz]{Corollary}
\newtheorem{remark}[satz]{Remark}

\def\Z{\mathbb {Z}}
\def\E{\mathsf {E}}

\def\F{\mathbb {F}}

\def\a{\alpha}

\def\G{\Gamma}
\def\a{\alpha}

\def\d{\delta}
\def\D{\Delta}

\def\({\big (}
\def\){\big )}

\def\ls{\leqslant}
\def\gs{\geqslant}

\def\_phi{\varphi}
\def\eps{\varepsilon}

\def\Gr{{\mathbf G}}
\def\FF{\widehat}

\def\tr{{\rm tr\,}}

\def\ge{\gs}
\def\le{\ls}
\def\la{\lambda}

\def\SL{{\rm SL}}
\def\U{{\rm U}}
\def\B{{\rm B}}

\begin{document}

\title{On a modular form of Zaremba's conjecture}

\author{Nikolay G. Moshchevitin and Ilya D. Shkredov}
\footnote{This work is supported by the Russian Science Foundation under grant 19--11--00001.}
\address[Nikolay Moshchevitin]{Lomonosov Moscow State University, Division of Mathematics, Moscow, Russia, and
	Steklov Mathematical Institute, ul. Gubkina, 8, Moscow, Russia, 119991}
\email{moshchevitin@gmail.com}
\address[Ilya Shkredov]{Steklov Mathematical Institute, ul. Gubkina, 8, Moscow, Russia, 119991, and
	IITP RAS, Bolshoy Karetny per. 19, Moscow, Russia, 127994, and
	MIPT, Institutskii per. 9, Dolgoprudnii, Russia, 141701}
\email{ilya.shkredov@gmail.com}

\begin{abstract}
	We prove that for any prime $p$ there is a divisible by $p$ number $q = O(p^{30})$ such that for a certain positive integer  $a$ coprime with $q$ the ratio $a/q$ has bounded partial quotients.   
	In the other direction  we show that there is an absolute constant $C>0$ such that for any prime $p$ exist divisible by $p$ number $q = O(p^{C})$ and a number $a$, $a$ coprime with $q$ such that all partial quotients of the ratio $a/q$ are bounded by two. 
\end{abstract}

\maketitle

\section{Introduction}

\newcommand{\zk}{\ensuremath{\mathfrak{k}}}

Let $a$ and $q$ be two positive coprime integers, $0<a<q$. 
By the Euclidean algorithm, a rational $a/q$ can be uniquely represented as a regular continued fraction
\begin{equation}\label{exe}
\frac{a}{q}=[0;b_1,\dots,b_s]=
\cfrac{1}{b_1 +\cfrac{1}{b_2 +\cfrac{1}{b_3+\cdots +\cfrac{1}{b_s}}}}
% \frac{1}{\displaystyle{b_1+\frac{1}{\displaystyle{b_2+
% \frac{1}{\displaystyle{b_3 +\dots +
% \displaystyle{\frac{1}{b_{s}}} }}}}}} 
,\qquad b_s \ge 2.
\end{equation}

Assuming $q$ is known, we use $b_j(a)$, $j=1,\ldots,s=s(a)$ to denote the partial quotients of $a/q$; that is,
\[
\frac aq := [ 0; b_1(a),\ldots,b_{s}(a)].
\]

Zaremba's famous conjecture \cite{zaremba1972methode} posits that there is an absolute constant $\zk$ with the following property:
for any positive integer $q$ there exists $a$ coprime to  $q$ such that in the continued fraction expansion (\ref{exe}) all partial quotients are bounded:
\[
b_j (a) \le \zk,\,\, 1\le j  \le s = s(a).
\]
In fact, Zaremba conjectured that $\zk=5$.
For large prime $q$, even $ \zk=2$ should be enough, as conjectured by Hensley \cite{hensley_SL2}, \cite{hensley1996}.
This theme is rather popular especially at the last time, see, e.g., 
papers 	
\cite{bourgain2011zarembas,bourgain2014zarembas},  \cite{FK}, \cite{hensley1989distribution}--\cite{hensley1996}, 
%\cite{KanIV}, \cite{Korobov}, \cite{Mosh_A+B}, 
\cite{KanIV}--\cite{Mosh_A+B}, 
\cite{Nied} and many others. 
The history of the question can be found, e.g.,  in \cite{NG_S}. 
Here we obtain the following "modular"\, version of Zaremba's conjecture. 
The first theorem  in this direction was  
%formulated and 
%obtained 
proved by Hensley 
in \cite{hensley_SL2} and after that in  \cite{MOW_MIX}, \cite{MOW_Schottky}.

\begin{theorem}
	There is an absolute constant $\zk$ such that for any prime number  $p$ 
	there exist some positive integers $q = O(p^{30})$,  $q\equiv 0 \pmod p$ and  $a$, $a$ coprime with $q$ having  the property that the ratio $a/q$ has partial quotients  bounded by  $\zk$.  
\label{t:main_intr}
\end{theorem}

Also, we can say something nontrivial about finite continued fractions with $\zk=2$. 
It differs our paper from \cite{bourgain2011zarembas}, \cite{bourgain2014zarembas}, \cite{KanIV}, \cite{MOW_MIX}, \cite{MOW_Schottky}.

\begin{theorem}
	There is an absolute constant $C>0$ such that for any prime number  $p$ 
	there exist some positive integers $q = O(p^{C})$, $q\equiv 0 \pmod p$ and  $a$, $a$ coprime with $q$  having  the property that the ratio $a/q$ has partial quotients  bounded by  $2$.  
	\label{t:main_intr2}
\end{theorem}

%The 
Our 
proof uses growth results in $\SL_2 (\F_p)$ and some well--known facts about the representation theory of $\SL_2 (\F_q)$. 
We study a combinatorial question about intersection of powers of a certain set of matrices $A \subseteq \SL_2 (\F_q)$  with an arbitrary   Borel subgroup and this seems like a new innovation.

In principle, results from \cite{hensley_SL2} can be written in a form  similar to Theorem \ref{t:main_intr} in an effective way but the dependence of $q$ on $p$ in 
%this paper 
\cite{hensley_SL2} 
is rather poor.  
Thus Theorem \ref{t:main_intr} can be considered as an explicit version (with very concrete constants)
of 
%previous 
Hensley's results as well as rather effective Theorem 2 from \cite{MOW_Schottky}.
Also, the methods of paper \cite{hensley_SL2} and papers \cite{MOW_MIX}, \cite{MOW_Schottky} are 
%completely 
very 
different from ours.  
%from \cite{hensley_SL2} where the first results in the direction were proved using completely different method. 

We thank I.D. Kan for  useful discussions and remarks.

\section{Definitions} 

Let $\Gr$ be a group with the identity $1$.
Given two sets $A,B\subset \Gr$, define  the \textit{product set}  of $A$ and $B$ as 
$$AB:=\{ab ~:~ a\in{A},\,b\in{B}\}\,.$$
In a similar way we define the higher product sets, e.g., $A^3$ is $AAA$. 
Let $A^{-1} := \{a^{-1} ~:~ a\in A \}$. 
The Ruzsa triangle inequality \cite{Ruz} says that
\[
	|C| |AB| \le |AC||C^{-1}B| 
\] 
for any sets $A,B,C \subseteq \Gr$. 
As usual, having two subsets $A,B$ of a group $\Gr$ denote by 
\[
	\E(A,B) = |\{ (a,a_1,b,b_1) \in A^2 \times B^2 ~:~ a^{-1} b = a^{-1}_1 b_1 \}| 
\]
the {\it common energy} of $A$ and $B$. 
Clearly, $\E(A,B) = \E(B,A)$ and by the Cauchy--Schwarz inequality 
\[
	\E(A,B) |A^{-1} B| \ge |A|^2 |B|^2 \,.
\]
We  use representation function notations like $r_{AB} (x)$ or $r_{AB^{-1}} (x)$, which counts the number of ways $x \in \Gr$ can be expressed as a product $ab$ or  $ab^{-1}$ with $a\in A$, $b\in B$, respectively. 
For example, $|A| = r_{AA^{-1}}(1)$ and  $\E (A,B) = r_{AA^{-1}BB^{-1}}(1) =\sum_x r^2_{A^{-1}B} (x)$. 
In this paper we use the same letter to denote a set $A\subseteq \Gr$ and  its characteristic function $A: \Gr \to \{0,1 \}$. 
We write $\F^*_q$ for $\F_q \setminus \{0\}$. 
The signs $\ll$ and $\gg$ are the usual Vinogradov symbols.
All logarithms are to base $2$.

\section{On the representation theory  of $\SL_2 (\F_p)$ and basis properties of its subsets}

First of all, we recall some notions and simple facts from the representation theory, see, e.g., \cite{Naimark} or \cite{Serr_representations}.
For a finite group $\Gr$ let $\FF{\Gr}$ be the set of all irreducible unitary representations of $\Gr$. 
It is well--known that size of $\FF{\Gr}$ coincides with  the number of all conjugate classes of $\Gr$.  
For $\rho \in \FF{\Gr}$ denote by $d_\rho$ the dimension of this representation. 
We write $\langle \cdot, \cdot \rangle$ for the corresponding  Hilbert--Schmidt scalar product 
$\langle A, B \rangle = \langle A, B \rangle_{HS}:= \tr (AB^*)$, where $A,B$ are any two matrices of the same sizes. 
Put $\| A\| = \sqrt{\langle A, A \rangle}$.
%$(d_\rho \times d_\rho)$--matrices. 
Clearly, $\langle \rho(g) A, \rho(g) B \rangle = \langle A, B \rangle$ and $\langle AX, Y\rangle = \langle X, A^* Y\rangle$.
Also, we have $\sum_{\rho \in \FF{\Gr}} d^2_\rho = |\Gr|$.

For any $f:\Gr \to \mathbb{C}$ and $\rho \in \FF{\Gr}$ define the matrix $\FF{f} (\rho)$, which is called the Fourier transform of $f$ at $\rho$ by the formula 
\begin{equation}\label{f:Fourier_representations}
\FF{f} (\rho) = \sum_{g\in \Gr} f(g) \rho (g) \,.
\end{equation}
Then the inverse formula takes place
\begin{equation}\label{f:inverse_representations}
f(g) = \frac{1}{|\Gr|} \sum_{\rho \in \FF{\Gr}} d_\rho \langle \FF{f} (\rho), \rho (g^{-1}) \rangle \,,
\end{equation}
and the Parseval identity is 
\begin{equation}\label{f:Parseval_representations}
\sum_{g\in \Gr} |f(g)|^2 = \frac{1}{|\Gr|} \sum_{\rho \in \FF{\Gr}} d_\rho \| \FF{f} (\rho) \|^2 \,.
\end{equation}
The main property of the Fourier transform is the convolution formula 
\begin{equation}\label{f:convolution_representations}
\FF{f*g} (\rho) = \FF{f} (\rho) \FF{g} (\rho) \,,
\end{equation}
where the convolution of two functions $f,g : \Gr \to \mathbb{C}$ is defined as 
\[
(f*g) (x) = \sum_{y\in \Gr} f(y) g(y^{-1}x) \,.
\]
Finally, it is easy to check that for any matrices $A,B$ one has $\| AB\| \le \| A\|_{o} \| B\|$ and $\| A\|_{o} \le \| A \|$, where  the operator $l^2$--norm  $\| A\|_{o}$ is just the absolute value of the maximal eigenvalue of $A$.  
In particular, it shows that $\| \cdot \|$ is indeed a matrix norm.

Now consider the group $\SL_2 (\F_q)$  of matrices 
\[
g=
\left( {\begin{array}{cc}
	a & b \\
	c & d \\
	\end{array} } \right) = (ab|cd) \,, \quad \quad a,b,c,d\in \F_q \,, \quad \quad ad-bc=1 \,.
\]
%which acts on $\F_p$ by 
%$$
%g z := \frac{az+b}{cz+d}  \,,\quad \quad z \in \F_p \,.
%$$
Clearly, $|\SL_2 (\F_q)| = q^3-q$.
Denote by $\B$ the standard  Borel subgroup of all upper--triangular matrices from  $\SL_2 (\F_q)$, by $\U \subset \B$ denote the standard unipotent subgroup of $\SL_2 (\F_q)$ of matrices $(1u|01)$, $u \in \F_q$  and by $\D \subset \B$ denote the subgroup of diagonal matrices. 
%It is well--known  \cite{Dickson} that 
$\B$ and all its conjugates form all maximal proper subgroups of $\SL_2 (\F_p)$.
Also, let $I_n$ be the identity matrix and $Z_n$ be the zero matrix of size $n\times n$. 
Detailed description of the representation theory of $\SL_2 (\F_q)$ can be found in \cite[Chapter II, Section 5]{Naimark}. 
We formulate the main result from  book \cite{Naimark} concerning this theme.

\begin{theorem}
	Let $q$ be an odd power. 
	There are $q+3$ nontrivial representations of $\SL_2 (\F_q)$, namely,\\\
$\bullet~$	$\frac{q-3}{2}$ representations $T_\chi$ of dimension $q+1$ indexed via $\frac{q-3}{2}$ nontrival multiplicative characters $\chi$ on $\F^*_q$, $\chi^2 \neq 1$,\\ 
$\bullet~$		a representation $\tilde{T}_1$ of dimension $q$,\\ 
$\bullet~$		two representations $T^{+}_{\chi_1}$, $T^{-}_{\chi_1}$ of dimension $\frac{q+1}{2}$, $\chi^2_1 = 1$, \\ 
$\bullet~$		two representations $S^{+}_{\pi_1}$, $S^{-}_{\pi_1}$ of dimension $\frac{q-1}{2}$, \\  
$\bullet~$		$\frac{q-1}{2}$ representations $S_\pi$ of dimension $q-1$ indexed via $\frac{q-1}{2}$ nontrival multiplicative characters $\pi$ on an arbitrary quadratic extension of $\F_q$, $\pi^2 \neq 1$.
\label{t:Naimark}
\end{theorem}

By $d_{\min}$, $d_{\max}$ denote the minimum/maximum over dimensions  of all nontrivial representations of a group $\Gr$. 
%$\SL_2 (\F_q)$. 
Thus the result above tells  us that in the case $\Gr = \SL_2 (\F_q)$ these quantities differ  just in two times roughly. 
Below we assume that $q\ge 3$.

Theorem \ref{t:Naimark} has two consequences, although, a slightly weaker result than Lemma \ref{l:A^n_large} can be obtained via the classical Theorem of Frobenius \cite{Frobenius}, see, e.g.,  \cite{sh_as}. Originally, similar arguments were suggested in \cite{SX}.

\begin{lemma}
	Let $n\ge 3$ be an integer, $A\subseteq \SL_2 (\F_q)$ be a set and $|A| \ge 2(q+1)^2 q^{2/n}$.
	Then $A^n = \SL_2 (\F_q)$. 
	Generally, if for some sets $X_1, \dots, X_n \subseteq \SL_2 (\F_q)$ one has 
	\[
		\prod_{j=1}^n |X_j| \ge (2q (q+1))^n (q-1)^{2} \,, 
	\]
	then $X_1 \dots X_n = \SL_2 (\F_q)$.  
\label{l:A^n_large}
\end{lemma}
\begin{proof} 
	Using formula \eqref{f:Parseval_representations} with $f=A$, we have for an arbitrary nontrivial representation $\rho$ that 
\begin{equation}\label{f:Fourier_est}
	\| A \|_{o} < \left(\frac{|A| |\SL_2 (\F_q)|}{d_{\min}} \right)^{1/2} = \left(\frac{|A| (q^3-q)}{d_{\min}} \right)^{1/2}  \,.
\end{equation}
	Hence for any $x\in \SL_2 (\F_q)$ we obtain  via formulae \eqref{f:inverse_representations}, \eqref{f:Parseval_representations} 
	and estimate \eqref{f:Fourier_est} that
\[
	A^n (x) > \frac{|A|^n}{|\SL_2 (\F_q)|} - \left(\frac{|A| (q^3-q)}{d_{\min}} \right)^{(n-2)/2} |A| \ge  0 \,,
\]
	provided $|A|^n \ge 2^{n-2} (q+1)^n q^n (q-1)^2$.
	% \ge 2 (q+1)^2 q^{2/n}$. 
	The second part of the lemma can be obtained similarly.
	This completes the proof. 
$\hfill\Box$
\end{proof}

\begin{remark}
	It is easy to see (or consult Lemma \ref{l:B_Wiener} below) that bound \eqref{f:Fourier_est} is sharp, e.g., take $A=\B$. 
\end{remark}

\bigskip

For any function $f : \Gr \to \mathbb{C}$ consider the  Wiener norm of $f$ defined as 
\begin{equation}\label{def:Wiener}	
	\| f\|_W := \frac{1}{|\Gr|} \sum_{\rho \in \FF{\Gr}} d_\rho \| \FF{f} (\rho) \| \,. 
\end{equation}

\begin{lemma}
	We have $\| \B\|_W = 1$.  
	Moreover,  $\| \FF{\B} (\tilde{T}_1) \| = \| \FF{\B} (\tilde{T}_1) \|_o = |\B|$ and the Fourier transform of $\B$ vanishes on all other nontrivial representations.  
\label{l:B_Wiener}
\end{lemma} 
\begin{proof}
	We 
	introduce 
	even three proofs  of upper and lower bounds of $\| \B\|_W$, although, the first and the third ones being shorter give slightly worse constants.
	Also, they do not provide full description of non--vanishing representations of $\B$.

	Since $\B$ is a subgroup, we see using  \eqref{f:Parseval_representations} twice that 
\[
	|\B|^2 = |\{b_1 b_2 = b_3 ~:~ b_1,b_2,b_3 \in \B\}| = \frac{1}{|\SL_2 (\F_q)|} \sum_{\rho \in \FF{\Gr}} d_\rho \langle \FF{\B}^2 (\rho), \FF{\B} (\rho) \rangle 
	\le
\]
\[ 
	\le
	\frac{1}{|\SL_2 (\F_q)|} \sum_{\rho} d_\rho \langle \FF{\B} (\rho), \FF{\B} (\rho) \rangle \| \FF{\B} (\rho)\|_{o} 
	\le
	 \frac{|\B|}{|\SL_2 (\F_q)|} \sum_{\rho} d_\rho \langle \FF{\B} (\rho), \FF{\B} (\rho) \rangle = |\B|^2 \,,
\]
	because, clearly, $\| \FF{\B} (\rho)\|_{o} \le |\B|$.  
	It means that for any representation $\rho$ either $\| \FF{\B} (\rho)\| = 0$ (and hence $\| \FF{\B} (\rho)\|_o = 0$) or $\| \FF{\B} (\rho)\|_{o} = |\B|$. 
	But another application of \eqref{f:Parseval_representations} gives us
\begin{equation}\label{tmp:01.10_1}
	|\B| = \frac{1}{|\SL_2 (\F_q)|} \sum_{\rho} d_\rho \|\FF{\B} (\rho) \|^2  
\end{equation}
	and hence the number $m$ of nontrivial representations $\rho$ such that $\| \FF{\B} (\rho)\|  \ge \| \FF{\B} (\rho)\|_{o} = |\B|$ is bounded in view of  Theorem \ref{t:Naimark} as 
\[
	|\B| \ge  \frac{|\B|^2}{|\SL_2 (\F_q)|} \left( 1 + \frac{m(q-1)}{2} \right) \,.
\]
	In other words, $m\le 2q/(q-1)$.
	Hence  
\begin{equation}\label{f:Borel_1-}
	\| \B\|_W \le \frac{|\B|}{|\SL_2 (\F_q)|} + \frac{m |\B|}{|\SL_2 (\F_q)|}\cdot d_{\max}  
	\le 
	\frac{|\B|}{|\SL_2 (\F_q)|} + \frac{2q (q+1) |\B|}{|\SL_2 (\F_q)|(q-1)} \le 4 \,.
\end{equation}
	A similar argument gives us a lower bound for $\| \B\|_W$ of the same sort.

	Let us give another proof which replaces $4$ to $1$ and uses the representation theory of $\SL_2 (\F_q)$ in a slightly more extensive way. 
	For $u_b \in \U$, $u_b=(1b|01)$, we have \cite[pages 121--123]{Naimark} that 
	in a certain orthogonal basis the following holds 
	$\tilde{T}_1 (u_b) = \mathrm{diag} (e(bj))$, $j=0,1,\dots,q-1$
	and for $g_\la = (\la 0|0 \la^{-1}) \in \D$ the matrix $\tilde{T}_1 (g_\la)$ is the direct sum of $I_1$ and a permutation matrix of size $(q-1) \times (q-1)$.
	Clearly, $\B = \D \U = \U \D$ and hence $\FF{\B} (\rho) = \FF{\D} (\rho) \FF{\U} (\rho)$ for any representation $\rho$.  
	But from above  $\FF{\U} (\tilde{T}_1)$ is the direct sum  $qI_1 \oplus Z_{q-1}$  and $\FF{\D} (\tilde{T}_1) = (q-1) I_1 \oplus 2\cdot J$, where 
	$J = (J_{ij})_{i,j=1}^{q-1}$ is a certain  $(q-1) \times (q-1)$ matrix with all components equal one for $i/j$ belonging to the set of quadratic residues   
%	(the last fact is not really important). 
	(such precise  description of $J$ is not really  important for us). 
	Hence  
	\[
		\FF{\B} (\tilde{T}_1) = \FF{\D} (\tilde{T}_1) \FF{\U} (\tilde{T}_1) = q(q-1) I_1 \oplus Z_{q-1} \,. 
	\]
	Thus $\| \FF{\B} (\tilde{T}_1) \| = \| \FF{\B} (\tilde{T}_1) \|_o = |\B|$. 
	Applying formula \eqref{tmp:01.10_1}, we obtain   
\begin{equation}\label{f:Borel_lower}
	|\B| \ge \frac{|\B|^2}{|\SL_2 (\F_q)|}  + \frac{q}{|\SL_2 (\F_q)|}  \|\FF{\B} (\tilde{T}_1) \|^2  
		= \frac{|\B|^2}{|\SL_2 (\F_q)|} (1 + q ) = |\B| \,.
\end{equation}
	It follows that for any other representations Fourier coefficients of $\B$ vanish. 
	Finally,
\begin{equation}\label{f:Borel_1}
	\| \B\|_W = \frac{|\B|}{|\SL_2 (\F_q)|} + \frac{q |\B|}{|\SL_2 (\F_q)|}  = 1 
\end{equation}
	as required.

	For the last proof it is enough to look at inequality \eqref{f:Borel_lower} and apply Theorem \ref{t:Naimark}, which gives that $\FF{\B}(T_\chi)$ must vanish thanks to  dimension of $T_\chi$.
	Further  
	if we have two nontrivial non--vanishing representations $S_\pi$ or $T^{\pm}_{\chi_1}$,  then it is again contradicts \eqref{f:Borel_lower} because  
	sum of their dimensions is  too large.
	Hence there is the only one nontrivial  non--vanishing representation (and calculations from the second proof show that it is indeed $\tilde{T}_1$) 
	or one of the following  pairs $(T^{\pm}_{\chi_1}, S^{\pm}_{\pi_1})$ or $(S^{+}_{\pi_1}, S^{-}_{\pi_1})$.    
	Thus a rough form of identity \eqref{f:Borel_1}, say, bound \eqref{f:Borel_1-} follows and, actually, we have not use any concrete basis in our first and the third
	% proof.  
	arguments. 
	This completes the proof of the lemma. 
$\hfill\Box$
\end{proof}

\begin{remark}
	One can show in the same way that an analogue of Lemma \ref{l:B_Wiener} takes place for any subgroup $\G$ of an arbitrary group $\Gr$, namely, 
	$\| \G\|_W \ll d_{\max}/d_{\min}$. 
\end{remark}

Lemma \ref{l:B_Wiener} gives us an alternative way to show that $A^3 \cap \B \neq \emptyset$.
% (we discuss the last question in details in the next Section). 
Indeed, just use estimate \eqref{f:Fourier_est} and write 
%applying the representation 
\[
	r_{A^3\B} (1) \ge \frac{|A|^3 |\B|}{|\SL_2 (\F_q)|} - \|\B \|_W \left(\frac{|A|(q^3-q)}{d_{\min}}\right)^{3/2} 
	=
	\frac{|A|^3 |\B|}{|\SL_2 (\F_q)|} - \left(\frac{|A|(q^3-q)}{d_{\min}}\right)^{3/2} > 0 \,,
\]
provided $|A| \gg q^{8/3}$. 
We improve this bound in the next section.

%One the other hand, without  exponent $2+2/n$ in Lemma \ref{l:A^n_large} cannot be significantly improved for large $n$. 
%Indeed, consider $x \B x^{-1}$, where $x$ is random. 
%Then estimating intersection of the expectation of $x \B x^{-1}$ and $\B$  (or see Lemma \ref{l:BgB} below), we obtain $|\B \cap x \B x^{-1} | \ll q$.
%% and hence $|A| \gg q^2$.
%Hence for any $n \ll q$ we have a set  $A \subseteq \B$ (so it does not contradicts the results of \cite{H}) 
%% for $ \B \cap x \B x^{-1}$ the following holds 
%such that $A^n \subseteq x \B x^{-1}$ and thus the intersection of $A^n$ and $\B$ is small.   

\section{On intersections of the product set with the Borel subgroup}

It was shown in the previous section (see Lemma \ref{l:A^n_large}) that for any $A \subseteq \SL_2 (\F_q)$ one has  $A^3 = \SL_2 (\F_p)$, provided $|A|^3 \gg q^8$ 
and in the same way 
%one can generalize slightly 
the last result 
%to 
holds for 
three different sets, namely, given  $X,Y,Z \subseteq \SL_2 (\F_q)$ 
%and obtain that 
with 
$|X||Y||Z| \gg q^8$, we have $XYZ=\SL_2 (\F_q)$. 
It is easy to see that in  this  generality the last result is sharp.
Indeed, let $X=S\B$, $Y=\B T$, where $S,T$ are two sets of sizes $\sqrt{q}/2$ which are chosen as $|X| \sim |S| |\B|$ and  $|Y| \sim |T| |\B|$
(e.g.,  take $S,T$ from left/right cosets of $\B$ thanks to the Bruhat decomposition).   
Then $XY=S\B T$, and hence $|XY| \le |S||T||B| \le |\SL_2 (\F_q)|/2$. 
Thus we take $Z^{-1}$ equals the complement to $XY$ in $\SL_2 (\F_q)$ and we see that the product set  $XYZ$ does not contain $1$  but $|X||Y||Z| \gg q^8$.

\bigskip 

Nevertheless, in  the "symmetric"\, case of the same set $A$ this $8/3$ bound can be improved, see Theorem \ref{t:8/3-c} below. 
We need a simple lemma and the proof of this result, as well as the proof of Theorem  \ref{t:8/3-c}
%, \ref{t:2-c} 
extensively play on non--commutative properties of $\SL_2 (\F_q)$.

\begin{lemma}
	Let $g\notin \B$ be a fixed element from $\SL_2 (\F_q)$. 
	Then for any $x$ one has 
\[
	r_{\B g\B} (x) \le q-1 \,.
\]
\label{l:BgB}
\end{lemma}
\begin{proof} 
	Let $g=(ab|cd)$ and $x=(\a \beta |\gamma \d)$. 
	By our assumption $c\neq 0$.
	% and $\gamma \neq 0$. 
	We have 
\begin{equation}\label{tmp:16.10_1}
\left( {\begin{array}{cc}
	\la & u \\
	0 & \la^{-1} \\
	\end{array} } \right)
\left( {\begin{array}{cc}
	a & b \\
	c & d \\
	\end{array} } \right) 
\left( {\begin{array}{cc}
	\mu & v \\
	0 & \mu^{-1} \\
	\end{array} } \right) 
=
\left( {\begin{array}{cc}
	(\la a + u c)\mu  & * \\
	\mu c/\la & vc/\la + d/(\la \mu) \\
	\end{array} } \right) 
=
	\left( {\begin{array}{cc}
	\a & \beta \\
	\gamma & \d \\
	\end{array} } \right) \,.
\end{equation}
	In other words, $\mu = \la \gamma c^{-1} \neq 0$  (hence $\gamma \neq 0$ automatically) and from  
\[
	\a = (\la a + u c)\mu = \la \gamma c^{-1} (\la a + u c) 
\]
	we see that having $\la$ we determine $u$ uniquely (then, equation \eqref{tmp:16.10_1} gives us $\mu, v$ automatically). 
	This completes the proof. 
$\hfill\Box$
\end{proof}

\bigskip 

Lemma \ref{l:BgB} quickly implies a result on the Bruhat decomposition of  $\SL_2 (\F_q)$. 

\begin{corollary}
	Let $g\in \SL_2 (\F_q) \setminus \B$. 
	Then $\B g\B = \SL_2 (\F_q) \setminus \B$.
\end{corollary}
\begin{proof}
	Clearly, $\B \cap \B g\B = \emptyset$ because $g\in \SL_2 (\F_q) \setminus \B$. 
	On the other hand, by the Cauchy--Schwartz inequality and Lemma \ref{l:BgB}, we have 
\[
	|\B g\B| \ge \frac{|\B|^4}{\E(\B,g\B)} \ge \frac{|\B|^4}{(q-1) |\B|^2} = q^3 - q^2 = |\SL_2 (\F_q) \setminus \B| \,.
\]
	This completes the proof. 
$\hfill\Box$
\end{proof}

\bigskip 

Using growth of products of $\B$ as in the last corollary, one can combinatorially improve the constant $8/3$ (to do this combine Lemma \ref{l:A^n_large} and bound \eqref{f:Borel_sum-product} below).
We suggest another method which uses the representation theory of $\SL_2 (\F_q)$ more extensively and which allows to improve this constant  further.

\begin{theorem}
	Let $A\subseteq \SL_2 (\F_q)$ be a set, 
	%$|A| \ge 2 q^{23/9}$. 
	$|A| \ge 4 q^{18/7}$. 
	Then 
	%either $A\cap \B \neq \emptyset$ or  
	$A^3 \cap \B \neq \emptyset$. 
	Generally, $A^n \cap \B \neq \emptyset$ provided $|A| \ge 4 q^{2+\frac{4}{3n-2}}$. 
\label{t:8/3-c}
\end{theorem}
\begin{proof}
	Let $g\notin \B$ and put $A^{\eps}_g = A^{\eps} \cap g\B$, where $\eps \in \{ 1,-1\}$. 
	Also, let $\D = \max_{\eps,\, g\notin \B} |A^\eps_g|$. 
	Since we can assume $A \cap \B = \emptyset$, it follows that 
\begin{equation}\label{tmp:04.10_1-}
	\E(A, \B) = \sum_{x} r^2_{A^{-1} \B} (x) = \sum_{x \notin \B} r^2_{A^{-1} \B} (x) \le \D |\B| |A| 
	%\,. 
\end{equation}
	and similarly for $\E(A^{-1}, \B)$. 
	On the other hand, from \eqref{tmp:04.10_1-} and by the second part of  Lemma \ref{l:B_Wiener}, we see that 
\begin{equation}\label{tmp:12.10_1}
	\D |\B| |A| \ge \E(A, \B) =  \frac{1}{|\SL_2 (\F_q)|} \sum_{\rho} d_\rho \| \FF{A}^* (\rho) \FF{\B} (\rho) \|^2 
	= 
	 \frac{q}{|\SL_2 (\F_q)|} \| \FF{A}^* (\tilde{T}_1) \FF{\B} (\tilde{T}_1) \|^2 
	 \,,
%	 := \sigma  \,.
\end{equation}
	and, again, similarly for $\| \FF{A} (\tilde{T}_1) \FF{\B} (\tilde{T}_1) \|^2$. 
%	Using the second part of  Lemma \ref{l:B_Wiener} again, we obtain 
%\[
%	\sigma = \frac{q |\B|^2}{|\SL_2 (\F_q)|} \langle \FF{A} (\tilde{T}_1)^* \FF{A} (\tilde{T}_1),  I_1 \oplus Z_{q-1} \rangle 
%	= 
%		\frac{q^2 (q-1)}{q+1}  \| \FF{A} (\tilde{T}_1) I_1 \oplus Z_{q-1} \|^2 \,.
%		%( \FF{A} (\tilde{T}_1)^* \FF{A} (\tilde{T}_1))_{11} 
%%		=
%%			\ge 
%%				 \frac{q^2 (q-1)}{q+1} \| \FF{A} (\tilde{T}_1) \|^2_o \,.
%\]
%	Hence
%\begin{equation}\label{tmp:04.10_1+}
%	\| \FF{A} (\tilde{T}_1) I_1 \oplus Z_{q-1} \|^2_o 
%	\le 2\D \frac{|A|(q+1)}{q} \,.
%\end{equation}
	Now consider the equation $b_1 a' a'' a b_2 = 1$ or, equivalently the equation $a'' a b_2 = (a')^{-1} b_1^{-1}$, where  $a, a',a'' \in A$ and $b_1, b_2 \in  \B$. 
	Clearly, if $A^3 \cap \B = \emptyset$, then this equation has no solutions.   
	Combining Lemma \ref{l:B_Wiener} with bound \eqref{tmp:12.10_1} and calculations as in the proof of Lemma \ref{l:A^n_large}, we see that this equation can be solved provided
\[
 	\frac{q}{|\SL_2 (\F_q)|} | \langle \FF{A}^2 (\tilde{T}_1) \FF{\B} (\tilde{T}_1), \FF{A}^* (\tilde{T}_1) \FF{\B}^* (\tilde{T}_1)  \rangle |
	\le \frac{q}{|\SL_2 (\F_q)|} \|\FF{A}^2 (\tilde{T}_1) \FF{\B} (\tilde{T}_1) \| 
		\cdot \| \FF{A}^* (\tilde{T}_1) \FF{\B} (\tilde{T}_1) \|
	\le 
\]
\[
	\le \frac{q}{|\SL_2 (\F_q)|} \|\FF{A} (\tilde{T}_1) \FF{\B} (\tilde{T}_1)  \| \|\FF{A}^* (\tilde{T}_1) \FF{\B} (\tilde{T}_1)  \| \| \FF{A} \|_o 
	\le
 \D |\B| |A|  \| \FF{A} \|_o  
 <
 \frac{|A|^3 |\B|^2}{|\SL_2 (\F_q)|}  \,.
%	\| \B\|_W \| \FF{A} (\tilde{T}_1) \|^3_o = \| \FF{A} (\tilde{T}_1) \|^3_o \,. 
\] 
	In other words, in view of \eqref{f:Fourier_est} 
	%we need 
	it is enough
	to have 
	%\eqref{tmp:04.10_1+}, we have
\begin{equation}\label{tmp:04.10_2-}
	|A|^4 \ge 2 (q+1)^2 \D^2 \cdot |A| q(q+1) 
\end{equation}
	or, equivalently,  
\begin{equation}\label{tmp:04.10_2}
	2 q (q+1)^3 \D^2  \le |A|^3 \,.
\end{equation}
	Now let us obtain another bound which works well when $\D$ is large. 
	Choose $g\notin \B$ and $\eps \in \{1,-1\}$ such that $\D = |A^\eps_g|$. 
	Using Lemma \ref{l:BgB}, we 
	%obtain
	derive 
\begin{equation}\label{f:B_Ag-}
\E(\B,A^\eps_g) = \sum_{x} r^2_{\B A^\eps_g} (x) \le \sum_{x} r_{\B A^\eps_g} (x) r_{\B g \B} (x) \le (q-1) |\B| |A^\eps_g| \,,
\end{equation}
and hence by the Cauchy--Schwarz inequality, we get 
\begin{equation}\label{f:B_Ag}
|\B A^\eps_g| \ge \frac{|\B|^2 |A^\eps_g|^2}{\E(\B,A^\eps_g)} \ge \frac{|\B| |A^\eps_g|}{q-1} = q \D \,.
\end{equation}	
	Consider the  equation $a_g (a' a'')^\eps =b$, where $b\in \B$, $a_g \in A^\eps_g$ and  $a',a'' \in A$.
	Clearly, if $A^3 \cap \B = \emptyset$, then this equation has no solutions.  
	To solve 
	%this equation 
	$a_g (a' a'')^\eps =b$ 
	it is enough to solve the equation $z  (a' a'')^\eps  = 1$, where now $z\in \B A^\eps_g$.
	%As in the proof of Lemma \ref{l:A^n_large} ( see the discussion at the beginning of this section), 
	Applying the second part of Lemma \ref{l:A^n_large} combining with \eqref{f:B_Ag}, 	we obtain that it is enough to have 
	\[
	8 q^3 (q+1)^3 (q-1)^{2} \le q\D |A|^2 \le |\B A^\eps_g| |A|^2 
	\]
	or, in other words,
	\begin{equation}\label{tmp:04.10_1}
	8 q^2 (q+1)^3 (q-1)^{2} \le \D |A|^2	\,.
	\end{equation}
	Considering the second power of  \eqref{tmp:04.10_1} and multiplying it with \eqref{tmp:04.10_2}, we obtain 
\[
	|A|^{7} \ge 2^{14} q^{18} \ge 2^7 q^5 (q+1)^{9} (q-1)^{4} 
	%\,.
\]
	as required.

	In the general case inequality \eqref{tmp:04.10_2} can be rewritten as 
\[
	|A|^n \ge 2^{n-2} \D^2 (q+1)^n q^{n-2} 
\]
	and using the second part of Lemma \ref{l:A^n_large}, we obtain an analogue of  \eqref{tmp:04.10_1} 
\[
	|A|^{n-1} \D \ge 2^n q^{n-1} (q+1)^n (q-1)^{2} \,.
\]
	Combining the last two bounds, we derive the required result. 
	This completes the proof. 
$\hfill\Box$
\end{proof}

%\begin{remark}
%	In the same way one can  obtain that $A^n \cap \B \neq \emptyset$ provided $|A| \gg q^{2+3/(2n-1)}$ which is better than the condition $|A| \gg q^{2+2/n}$ of Lemma \ref{l:A^n_large}.  
%\end{remark}

\begin{remark}
	It is easy to see that Theorem \ref{t:8/3-c}, as well as Lemma \ref{l:BgB} (and also Lemma \ref{l:B_Wiener}) take place for any Borel subgroup not just for the standard one. 
\label{r:BgB_B*}
\end{remark}

\begin{remark}
	It is easy to see that the arguments of the proof of Theorem \ref{t:8/3-c}  
	give the following combinatorial statement about left/right multiplication of an arbitrary set $A$ by $\B$
	(just combine  bounds \eqref{tmp:04.10_1-} and \eqref{f:B_Ag}), namely,
\begin{equation}\label{f:Borel_sum-product}
	\max\{ |A\B|, |\B A| \} \gg \min\{q^{3/2} |A|^{1/2}, |A|^2 q^{-2} \} \,.
\end{equation} 
\end{remark}

\bigskip

As we have seen by 
%Thus by 
%the previous result 
Theorem \ref{t:8/3-c}
%one can obtain similar results on intersections of $A^n$ and $\B$ for large $n$ but all they have restriction $|A| \gg q^{2+\eps_n}$, where $\eps_n \to 0$ as $n\to \infty$. 
%gives us results on intersections of 
we know that 
$A^n \cap \B \neq \emptyset$ for large $n$ but under the condition $|A| \gg q^{2+\eps}$ for a certain $\eps>0$.
% where $\eps_n \to 0$ as $n\to \infty$. 
For the purpose of the next section we need to break the described  $q^2$--barrier and we do this for  prime $q$, using growth in $\SL_2 (\F_p)$.  
Let us recall quickly what is known about growth of generating sets in $\SL_2 (\F_p)$.
%this group. 
%$\SL_2 (\F_p)$.
% for prime $q$.  
In paper \cite{H} Helfgott  obtained his famous result in this direction 
% on growth in $\SL_2 (\F_p)$  
and  we proved in \cite{RS_SL2} the following form of Helfgott's result.

\begin{theorem}
	Let $A \subseteq \SL_2 (\F_p)$ be a set, $A=A^{-1}$ which generates the whole group. 
	Then $|AAA| \gg |A|^{1+1/20}$. 
	\label{t:Misha_Je}
\end{theorem}

Thus in the case of an arbitrary  symmetric generating set and a prime number $p$ 
Theorem \ref{t:Misha_Je},  combining with Theorem \ref{t:8/3-c}, 
allow to obtain some  bounds which guarantee that $A^n = \SL_2 (\F_p)$.
%(of course they are worse that which are given by Theorem \ref{t:2-c}).
%Moreover 
For example, if $A$ generates $\SL_2(\F_p)$, $A=A^{-1}$, 
%$A\cap \B = \emptyset$ 
and $|A| \gg p^{2-\epsilon}$, $\epsilon < \frac{2}{21}$, 
then $A^n \cap \B \neq \emptyset$ for $n\ge \frac{84-42\epsilon}{2-21\epsilon}$. 
On the other hand, 
the methods from \cite{H}, \cite{RS_SL2} allow to obtain the following result about generation of $\SL_2 (\F_p)$ via large and not necessary symmetric sets
(the condition of non--symmetricity of $A$ is rather crucial for us, see the next section).
%We give the scheme of the proof, full details can be found in \cite{RS_SL2}. 

%For $g\in \SL_2(\F_q)$ let 
%\[
%\mathcal{C}_g = \{s \in \SL_2 (\F_p) ~:~ \tr(s) = \tr(g) \} \,.
%\]

\begin{theorem}
	Let $A\subseteq \SL_2 (\F_p)$ be a generating set, $p\ge 5$ and $|A| \gg p^{2-\epsilon}$, $\epsilon < \frac{2}{25}$. 
	Then $A^n \cap \B \neq \emptyset$ for $n\ge \frac{100-50\epsilon}{2-25 \epsilon}$.  
	Also, $A^n = \SL_2 (\F_p)$, provided $n\ge \frac{144}{2-25 \epsilon}$. 
\label{t:Misha_Je_large}
\end{theorem}
\begin{proof}
	Put $K=|AAA|/|A|$.
	We can assume that, say, $|A| \le p^{2+2/35}$ because otherwise one can apply Theorem \ref{t:8/3-c}.  
	We call an element $g\in \SL_2 (\F_p)$ to be regular if $\tr(g) \neq 0, \pm 2$ and let $\mathcal{C}_g$ be the correspondent conjugate class, namely, 
%	For $g\in \SL_2(\F_q)$ let 
	\[
	\mathcal{C}_g = \{s \in \SL_2 (\F_p) ~:~ \tr(s) = \tr(g) \} \,.
	\]
	Let $T$ be a maximal torus (in $\SL_2 (\F_p)$ it is just a maximal commutative subgroup) such that there is $g\in T\cap A^{-1}A$ and $g \neq 1$.
	% is regular.
	By \cite[Lemma 5]{RS_SL2} such torus $T_*$, containing a regular element $g$,  exists, 
	%with $\tr(g) \neq 0$, exists, 
	otherwise $K\gg |A|^{2/3}$.  
	Firstly, suppose that for a certain $h\in A$ the torus $T'=hTh^{-1}$ has no such property, i.e., there are no 
	%regular 
	nontrivial 
	elements from $A^{-1}A \cap T'$. 
	Then for the  element $g'=hgh^{-1} \in T'$ (in the case $T=T_*$ the element $g'$ is regular) the projection $a\to ag'a^{-1}$, $a\in A$ is one--to--one.   
	Hence $|A^2 A^{-1} A A^{-2} \cap \mathcal{C}_g| \ge |A|$. 
	By \cite[Lemma 11]{RS_SL2}, we have $|S \cap \mathcal{C}_g| \ll |S^{-1}S|^{2/3} + p$ for any set $S$ and regular $g$. 
	Using the Ruzsa triangle inequality, we obtain 
	%\cite{Ruz}, we obtain 
\begin{equation}\label{tmp:15.10_I}
	|(A^2 A^{-1} A A^{-2})^{-1}(A^2 A^{-1} A A^{-2})| \le |A|^{-1} |A^2 A^{-1} A A^{-3}| |A^3 A^{-1} A A^{-2}|
	=
\end{equation}
%\begin{equation}\label{tmp:15.10_II} 
\[
	=
	|A|^{-1} |A^3 A^{-1} A A^{-2}|^2 \le |A|^{-1} (|A|^{-1} |A^3 A^{-2}| |A^2 A^{-2}| )^2 \le |A|^{-1} (|A|^{-3} |A^4| |A^3|^3 )^2 \le K^{12} |A|
\]
%\end{equation}
	and hence
\[
	|A| \ll |(A^2 A^{-1} A A^{-2})^{-1}(A^2 A^{-1} A A^{-2})|^{2/3} + p \ll K^{8} |A|^{2/3} \,.
\] 
	It gives us $K\gg |A|^{1/24}$.

	%Now 
	In the complementary second case  (see \cite{RS_SL2}) thanks to the fact that $A$ is a generating set, we  
	suppose that for {\it any} $h\in \SL_2(\F_p)$ there is a 
	%regular 
	nontrivial 
	element from $A^{-1}A$ belonging to the torus $hTh^{-1}$. 
	Then $A^{-1}A$ is partitioned between these tori and hence again by \cite[Lemma 11]{RS_SL2}, as well as the Ruzsa triangle inequality, we obtain 
\[
	|(AA^{-1}AA^{-1})^{-1}(AA^{-1}AA^{-1})| \le |A|^{-1} |A^2 A^{-1} A A^{-1}|^2 
	\le
\]
\[ 
	\le
		|A|^{-1} (|A|^{-1} |A^2 A^{-2}| |A^2 A^{-1}|)^2 
		\le |A|^{-1} (|A|^{-3} |A^3|^4)^2    \le K^8 |A| 
\]
	and whence
\[
	K^2 |A| \ge |A^{-1}A| \ge \sum_{h \in \SL_2 (\F_p)/N(T_*)} |A^{-1} A \cap h T_* h^{-1}| 
	\gg
\]
\[ 
	\gg
		p^2 \cdot \frac{|A|}{|(AA^{-1}AA^{-1})^{-1}(AA^{-1}AA^{-1})|^{2/3}} 
	\ge p^2 |A|^{1/3} K^{-16/3} \,,
\]
	where $N(T)$ is the normalizer of any torus $T$, $|N(T)| \asymp |T| \asymp p$.  
	Hence thanks to our assumption $|A| \le p^{2+2/35}$, we have $K \gg p^{3/11} |A|^{-1/11} \gg |A|^{1/24}$. 
	In other words, we always obtain  $|AAA| \gg p^{2+\frac{2-25\epsilon}{24}}$.
	After that apply Theorem \ref{t:8/3-c} to find that $A^n \cap \B \neq \emptyset$ for $n\ge \frac{100-50\epsilon}{2-25 \epsilon}$. 	
	If we use Lemma \ref{l:A^n_large} instead of Theorem \ref{t:8/3-c}, then we obtain  $A^n = \SL_2 (\F_p)$, provided $n\ge \frac{144}{2-25 \epsilon}$. 
	This completes the proof. 
	$\hfill\Box$
\end{proof}

%\begin{remark}
%	In Theorem \ref{t:Misha_Je_large} the condition that $A$ generates $\SL_2 (\F_p)$ is necessary even for small $\epsilon$.  
%	Indeed, if not, then for $\epsilon <1$, we see that $A\subseteq \B_* := c\B c^{-1}$ for a certain $c\in \SL_2(\F_p) \setminus \B$, see \cite[Theorem 6.17, Theorem 6.25]{Dickson} for the structure of subgroups of  $\SL_2(\F_p)$. 
%%It gives, first of all, in view of Lemma \ref{l:BgB} that $|A\B|$, $|\B A| \gg p |A|$. 
%	Nevertheless, let $A$ be a subgroup 
%\end{remark}

\bigskip

Thus for sufficiently small $\epsilon>0$ one can take $n=51$ to get  $A^n \cap \B \neq \emptyset$ (and $n=73$ to obtain $A^n  = \SL_2 (\F_p)$).
In the next section we improve this bound for a special set $A$ but nevertheless the arguments of the proof of Theorem \ref{t:Misha_Je_large} will be used in the proof of Theorem \ref{t:main_intr2} from the Introduction.

\bigskip

We finish this section showing that generating sets $A$ of sizes close to $p^2$ (actually, the condition $|A| =\Omega(p^{3/2+\eps})$ is enough) with small tripling constant $K=|A^3|/|A|$ avoid all Borel subgroups.

\begin{lemma}
	Let $A\subseteq \SL_2 (\F_p)$ be a generating set, $p\ge 5$ and $K=|A^3|/|A|$.
%	$|A| \gg p^{2-\epsilon}$, $\epsilon < \frac{2}{25}$. 
Then for any Borel subgroup $\B_*$ one has $|A\cap \B_*| \le 2p K^{5/3} |A|^{1/3}$.    	
\label{l:A_cap_B}
\end{lemma}
\begin{proof}
	We obtain the result for the standard Borel subgroup $\B$ and after that apply the conjugation to prove our Lemma in full generality. 
	Let $\gamma \in \F_p^*$ be any number and  $l_\gamma$ be the line 
	$$ l_\gamma = \{ (\gamma u| 0 \gamma^{-1}) ~:~ u\in \F_p \} \subset \SL_2 (\F_p) \,.$$
	By \cite[Lemma 7]{RS_SL2}, we have $|A\cap l_\gamma| \le 2 |A^3 A^{-1} A|^{1/3}$.
	Using the last bound, as well as the Ruzsa triangle inequality, we obtain
\[
	|A\cap \B| \le \sum_{\gamma \in \F^*_p} |A\cap l_\gamma| \le 2 p |A^3 A^{-1} A|^{1/3} 
		\le 
			2p (|A^4||A^{-2} A|/|A|)^{1/3}
				\le	
				2p K^{5/3} |A|^{1/3} \,.
\] 
	This completes the proof. 
$\hfill\Box$
\end{proof}

\begin{remark}
	Examining the proof of Lemma 7 from \cite{RS_SL2} one can equally write $|A\cap l_\gamma| \le 2 |A^3 A^{-2}|^{1/3}$ and hence by the calculations above 
	$|A\cap \B_*| \le 2p K^{4/3} |A|^{1/3}$. 
	Nevertheless, his better estimate has no influence to the final bound in Theorem \ref{t:main_intr}. 
\end{remark}

\begin{remark}
	Bounds for intersections of $A\subseteq \SL_2 (\F_q)$, $K=|A^3|/|A|$ with $g\B_*$, where $g\notin \B_*$ are much simpler and follow from Lemma \ref{l:BgB} (also, see Remark \ref{r:BgB_B*}).
	Indeed, by this result putting $A_* = A\cap g\B_*$, we have 
\[
	K|A|\ge |AA| \ge |A_* A_*| \ge \frac{|A_*|^4}{\E(A^{-1}_*, A_*)} \ge \frac{|A_*|^4}{\E(A^{-1}_*,g\B_*)} \ge \frac{|A_*|^2}{q-1}
\]
without any assumptions on generating properties of $A$. 
\label{r:sB}
\end{remark}

%\section{On modular Zaremba's conjecture} 
\section{On Zaremba's conjecture}

%Now 
In this section we 
apply methods of the proofs of Theorems \ref{t:8/3-c}, \ref{t:Misha_Je_large} to Zaremba conjecture but also we use the specific of this problem, i.e. the special form of the correspondent set of matrices from  $\SL_2 (\F_p)$.

\bigskip 

Denote by $F_M(Q)$  the set of all {\it rational} numbers  $\frac{u}{v}, (u,v) = 1$ from $[0,1]$ with all partial quotients in (\ref{exe}) not exceeding $M$ and with $ v\le Q$:
\[
F_M(Q)=\left\{
\frac uv=[0;b_1,\ldots,b_s]\colon (u,v)=1, 0\leq u\leq v\leq Q,\, b_1,\ldots,b_s\leq M
\right\} \,.
\]
By $F_M$ denote the set of all {\it irrational} 
%real 
numbers  from $[0,1]$ with partial quotients less than or equal to $M$.
From \cite{hensley1992continued} we know that the Hausdorff dimension $w_M$ of the set  $F_M$ satisfies
\begin{equation}
w_M = 1- \frac{6}{\pi^2}\frac{1}{M} -
\frac{72}{\pi^4}\frac{\log M}{M^2} + O\left(\frac{1}{M^2}\right),
\,\,\, M \to \infty \,,
\label{HHD}
\end{equation}
however here we need a simpler result from \cite{hensley1989distribution}, which states that
\begin{equation}\label{oop}
1-w_{M} \asymp  \frac{1}{M}
\end{equation}
with absolute constants in the sign $\asymp$.
Explicit estimates for dimensions of $F_M$  for certain values of $M$ can be found in \cite{jenkinson2004density}, \cite{JP} and in other papers.
For example, see \cite{JP} 
\[
	w_2 = 0.5312805062772051416244686... 
\]
In  papers \cite{hensley1989distribution,hensley1990distribution} Hensley gives the bound
\begin{equation}
|F_M(Q)| \asymp_M Q^{2w_M} \,.
\label{QLOW}
\end{equation}

\bigskip

Now we are ready to prove Theorem \ref{t:main_intr} from the Introduction. 
One has 
\begin{equation}\label{f:continuants_aj}
\left( {\begin{array}{cc}
	0 & 1 \\
	1 & b_1 \\
	\end{array} } \right) 
	\dots 
\left( {\begin{array}{cc}
	0 & 1 \\
	1 & b_s \\
	\end{array} } \right) 
=
\left( {\begin{array}{cc}
	p_{s-1} & p_s \\
	q_{s-1} & q_s \\
	\end{array} } \right) \,,
\end{equation}
where $p_s/q_s =[0;b_1,\dots, b_s]$ and $p_{s-1}/q_{s-1} =[0;b_1,\dots, b_{s-1}]$.
Clearly, $p_{s-1} q_s - p_s q_{s-1} = (-1)^{s}$.
Let $Q=p-1$ and consider the set $F_M(Q)$. 
Any $u/v \in F_M(Q)$ corresponds to a matrix from \eqref{f:continuants_aj} such that  $b_j \le M$.
The set $F_M(Q)$ splits into ratios with even $s$ and with odd $s$, 
%i.e.,  
in other words 
$F_M(Q) =  F^{even}_M(Q) \bigsqcup F^{odd}_M(Q)$. 
%and we take the largest set
%(actually, one can easily show via \eqref{QLOW} that these sets have comparable sizes up to constants).  
Let $A \subseteq \SL_2 (\F_p)$ be the set of matrices of the form above with even $s$.
It is easy to see from \eqref{QLOW}, multiplying if it is needed the set $F^{odd}_M (Q)$ by $(01|1b)^{-1}$, $1\le b \le M$ that 
$|F^{even}_M(Q)| \gg_M |F_M (Q)|  \gg_M Q^{2w_M}$.
%if the largest set is  $F^{even}_M(Q)$, otherwise let $A$ be matrices of the form above with odd $s$, multiplied by $(01|10)$ (this corresponds to permutation of columns of our matrix from \eqref{f:continuants_aj}). 
%Anyway $|A| \ge |F_M(Q)|/2$ and $A \subseteq \SL_2 (\F_p)$.
It is easy to check that  if for a certain $n$ one has $A^n \cap \B \neq \emptyset$, then $q_{s-1}$ 
%or $q_{s}$ 
equals zero modulo $p$ and hence there is $u/v \in F_M ((2p)^n)$ such that $v\equiv 0 \pmod p$.
%Also, notice that 
%In particular, 
In a similar way, 
we can easily assume that for any $g = (ab|cd)\in A$ all entries $a,b,c,d$ are nonzero (and hence by the construction they are nonzero modulo $p$), see, e.g., \cite[page 46]{hensley_SL2} or the proof of Lemma \ref{l:M^3} below (the same paper \cite{hensley_SL2} contains the fact that $A$ is a generating subset of $\SL_2 (\F_p)$). 
Analogously, we can suppose that all $g \in A$ are regular, that is, $\tr(g) \neq 0,\pm 2$. 
%Now put $A=A_* \cup A^{-1}_*$. 
%Then $A=A^{-1}$ and it is easy to see that if for a certain $n$ one has $A^n \cap \B \neq \emptyset$, then either $q_s$ or $q_{s-1}$ equals zero modulo $p$ and hence there is $u/v \in F_M (p^n)$ such that $v\equiv 0 \pmod p$.
Let $K = |AAA|/|A|$ and $\tilde{K} = |AA|/|A| = K^\a$, $0\le \a \le 1$. 
%Clearly, $\tilde{K} \le K$. 

We need to 
%bound 
estimate 
from below cardinality of the set of all possible traces of $A$, that is, cardinality of the set of sums $q_{s} + p_{s-1}$ 
(this expression is called "cyclical continuant").  
Fix $p_{s-1}$ and  $q_s$.
% (the case of fixed $(q_{s-1}, p_s)$ can be considered similarly). 
Then $p_{s-1} q_s - 1 = p_s q_{s-1}$ and thus $p_s$ is a divisor of $p_{s-1} q_s - 1$. 
In particular, the number of such $p_s$ is at most $p^\eps$ for any $\eps>0$. 
But now knowing the pair $(p_s,q_s)$, we  determine   the correspondent matrix \eqref{f:continuants_aj} from $A$ uniquely. 
Hence the number of different pairs $(p_{s-1}, q_s)$ is at least $\Omega_M (p^{-\eps}|F_M(Q)|)$ 
and thus the number of different traces of all matrices from  $A$ is $\Omega_M (p^{-1-\eps} |A|)$. 
Actually, one can improve the last bound to   $\Omega_M (p^{-1} |A|)$.

\begin{lemma}
	The number of all possible sums $q_{s} + p_{s-1}$ is at least $\Omega(|A|/(M^3 p))$. 
\label{l:M^3}
\end{lemma}
\begin{proof} 
	As above fix 
	%$(p_{s-1}, q_s) \in A$.
	$q_s$ and $p_{s-1}$. 
	It is well--known (see, e.g., \cite{hensley_SL2}) that 
	$q_s = \langle b_1,\dots, b_s \rangle$, 
	$p_s = \langle b_2,\dots, b_s \rangle$, 
	$q_{s-1} = \langle b_1,\dots, b_{s-1} \rangle$, 
	$p_{s-1} = \langle b_2,\dots, b_{s-1} \rangle$, where by $\langle x_1,\dots, x_n \rangle$ we have denoted the corresponding continuant.  
	We know that 
\begin{equation}\label{tmp:01.11_1}
	-p_{s} q_{s-1} = - q_s p_{s-1} + 1 \,.
\end{equation}
	Substituting  the well--known formula $p_{s} = b_s p_{s-1} + p_{s-2}$ into \eqref{tmp:01.11_1}, we obtain  
\begin{equation}\label{tmp:01.11_2}
	-b_s p_{s-1} q_{s-1} \equiv  - q_s p_{s-1} + 1 \pmod {p_{s-2}} \,.
\end{equation}
	and thus for any fixed $b_s \neq 0 \pmod {p_{s-2}}$ the number $q_{s-1}$ is uniquely determined modulo $p_{s-2} = \langle b_2,\dots, b_{s-2} \rangle $.  
	But applying the recurrence formula for continuants again, we get 
\[
	q_{s-1} = b_{s-1} \langle b_1,\dots, b_{s-2} \rangle + \langle b_1,\dots, b_{s-3} \rangle \le (b_{s-1} + 1) \langle b_1,\dots, b_{s-2} \rangle =
\]
\[
	= (b_{s-1} + 1) ( b_1 p_{s-2}  + \langle b_3,\dots, b_{s-2} \rangle) \le (b_{s-1} + 1) (b_1+1) p_{s-2} \,.
\]
%	Hence fixing $b_1,b_{s-1}$ in at most $M^2$ ways, we see that
	It follows that  there are at most $(M+1)^2$ possibilities for $q_{s-1}$. 
	Now if $b_s \equiv 0 \pmod {p_{s-2}}$, then $M\ge b_s \ge p_{s-2} \ge \left(\frac{1+\sqrt{5}}{2}\right)^{s-2}$  and hence $s\ll \log M$. 
	It gives us, say, at most $M^s \ll M^{O(1)} \le |A|/2$   
	%|A|/(M^3 p)$ 
	matrices from $A$.
	% and this is negligible.  
	This completes the proof of the lemma. 
$\hfill\Box$
\end{proof}

\bigskip

Now recall \cite[Lemma 12]{RS_SL2}, which is a variant of the Helfgott map \cite{H} from  \cite{Brendan_rich} (we have already used similar arguments in the proof of Theorem \ref{t:Misha_Je_large}). 
For the sake of the completeness we give the proof of a "statistical"\, 
%variant 
version 
of this result.

\begin{lemma}
	Let $\Gr$ be any group and $A\subseteq \Gr$ be a finite set.
	Then for an arbitrary $g\in \Gr$,  there is $A_0 \subseteq A$, $|A_0| \ge |A|/2$ such that for any $a_0 \in A_0$ the following holds 
	%one has
	\begin{equation}\label{f:CS_ineq}
	|A|/2 \le |{\rm Conj} (g) \cap AgA^{-1}| \cdot |{\rm Centr}(g) \cap a_0^{-1} A| \,.
	\end{equation}
	Here ${\rm Conj} (g)$ is the conjugacy class and ${\rm Centr}(g)$ is the centrlizer of $g$ in $\Gr$.
	\label{l:CS_ineq}
\end{lemma}
\begin{proof} 
	Let $\_phi : A \to  {\rm Conj} (g) \cap AgA^{-1}$ be the Helfgott map $\_phi(a) := a g a^{-1}$. 
	One sees that $\_phi(a) = \_phi (b)$ iff 
	\[
	b^{-1} a g = g b^{-1} a \,.
	\]
	In other words, $b^{-1} a \in {\rm Centr}(g) \cap A^{-1} A$. 
	Clearly, then
	\[
	|A| = \sum_{c\in {\rm Conj} (g) \cap AgA^{-1}} |\{ a\in A ~:~ \_phi (a) = c\}| 
	\le
	\]
	\begin{equation}\label{tmp:01.11_10}
	\le 
		2 \sum_{c\in {\rm Conj} (g) \cap AgA^{-1} ~:~ |\{ a\in A ~:~ \_phi (a) = c\}| \ge |A|/(2|{\rm Conj} (g) \cap AgA^{-1}|)} |\{ a\in A ~:~ \_phi (a) = c\}|	\,.
	\end{equation}
	For $c\in \_phi (A) \subseteq {\rm Conj} (g) \cap AgA^{-1}$ put $A (c) = \_phi^{-1} (c) \subseteq A$ and let 
	$$
		A_0 = \bigsqcup_{c ~:~ |A (c)| \ge |A|/(2|{\rm Conj} (g) \cap AgA^{-1}|)} A (c) \,.
	$$
	In other words, estimate \eqref{tmp:01.11_10} gives us
\[
	|A_0| = \sum_c |A (c)| \ge |A|/2 \,.
%	 |\{ a\in A ~:~ |\{ a\in A ~:~ \_phi (a) = c\}| \ge |A|/(2|{\rm Conj} (g) \cap AgA^{-1}|) \}| \ge |A|/2 \,.
\]
	But for any $b\in A_0$ one has $|{\rm Centr}(g) \cap b^{-1} A| \ge  |A|/(2|{\rm Conj} (g) \cap AgA^{-1}|)$ 
	as required. 
	This completes the proof of the lemma. 
$\hfill\Box$
\end{proof}

\bigskip 

Now summing inequality \eqref{f:CS_ineq} over all  $g\in A$ with different traces, we obtain in view of the Ruzsa triangle inequality and Lemma \ref{l:M^3} that 
\begin{equation}\label{tmp:15.10_1}
	|A|^2 p^{-1} \ll_M |AAA^{-1}| \cdot \max_{g\in A}  |{\rm Centr}(g) \cap a_0^{-1}(g) A| 
		\le	
			K \tilde{K} |A| \cdot \max_{g\in A}  |{\rm Centr}(g) \cap a_0^{-1}(g) A|  \,.
\end{equation}
	Here for every $g\in A$ we have taken a concrete $a_0 (g) \in A_0 (g)$ but in view of  Lemma \ref{l:CS_ineq} it is known that there are a lot of them and we will use this fact a little bit later. 
	Now by \cite[Lemma 4.7]{H}, we see 
	%know 
	that 
\[
	%|A^2 A^{-1} A A^{-2}| \ge 
	|(a_0^{-1}(g) A) g_* (a_0^{-1}(g) A) g^{-1}_* (a_0^{-1}(g) A)^{-1}| \gg |{\rm Centr}(g) \cap a_0^{-1}(g) A|^3 \,,
\]
	where $g_* = (ab|cd)$ is any element from $A$ such that $abcd \neq 0$ in the basis where $g$ has the diagonal form.   
	Thanks to Lemma \ref{l:A_cap_B} and Remark \ref{r:sB} we can choose $g_* = a_0 (g)$, otherwise $|A| \ll p^{3/2} K^{5/2}$. 
	In the last case if, say, $|A| \gg p^{2-1/35}$, then $K\gg p^{33/175}$ and hence $|A^3| \gg p^{2+4/25}$. 
	Using Theorem \ref{t:8/3-c}, we see that one can take $n=27$  and this is better than we want to prove. 
	Then with this choice of $g_*$, we have  by the Ruzsa triangle inequality 
\[
	|A^2 g^{-1}_* A^{-1}| \le |A^2 A^{-2}| \le K^2 |A| \,, 
\]
%	As in \eqref{tmp:15.10_I}, \eqref{tmp:15.10_II}, we get
%\[
%	|A^2 A^{-1} A A^{-2}| \le |A|^{-1} |A^2 A^{-2}|^2 \le |A|^{-3} |AAA|^4 \le K^4 |A| \,, 
%\]
	and hence $|{\rm Centr}(g) \cap a_0^{-1}(g) A| \ll K^{2/3} |A|^{1/3}$. 
%	Returning to \eqref{tmp:15.10_1} and 
	Substituting the last bound into \eqref{tmp:15.10_1}, we get 
\begin{equation}\label{tmp:15.10_1-}
	|A|^2 p^{-1} \ll_M K \tilde{K} |A| \cdot  K^{2/3} |A|^{1/3}
\end{equation}
	and hence 
\begin{equation}\label{tmp:15.10_2}
	K \gg_M (|A|^2 p^{-3})^{\frac{1}{5+3\a}} \gg p^{\frac{4w_M}{5+3\a} - \frac{3}{5+3\a}} \,.
\end{equation}
	In other words, $|AAA| \gg_M p^{2+\frac{w_M (14+6\a) - 13- 6\a}{5+3\a}}$. 
	Take $M$ sufficiently large 
	%and $\eps$ sufficiently small 
	such that $w_M (14+6\a) - 13- 6\a  >0$. 
	Using Theorem \ref{t:8/3-c}, we see that for any 
\begin{equation}\label{f:n_1}
	n\ge \frac{w_M (28+12\a)- 6}{w_M (14+6\a) - 13- 6\a} 
\end{equation} 
	one has  $A^n \cap \B \neq \emptyset$. 
	On the other hand, from \eqref{tmp:15.10_2}, we get
\[
	|AA| = |A|K^\a \gg p^{2+ \frac{w_M(10+10\a) -10 - 9 \a}{5+3\a}} \,.
\]
	Suppose that $w_M(10+10\a) -10 - 9 \a > 0$. 
	It can be done if $\a>0$ and if we take  sufficiently large $M$.
	%(and $\eps$ sufficiently small).
	Applying Theorem \ref{t:8/3-c} one more time, we derive that for any 
\begin{equation}\label{f:n_2}
	n \ge \frac{2}{3} \cdot \frac{w_M(20+20\a) - 6 \a}{w_M(10+10\a) -10 - 9 \a} 
\end{equation}
	one has  $A^n \cap \B \neq \emptyset$. 
	Comparing \eqref{f:n_1} and \eqref{f:n_2}, we choose $\a$ optimally when
\[
	\a^2 (120 w^2_M - 12w_M - 72) + \a(400 w_M^2 -368 w_M + 6) 
%	+
%\]
%\[ 
	+
	280 w_M^2 + 180 - 500 w_M = 0 
\]
	and it gives 
\[
	18 \a^2 + 19 \a - 20 = 0
\]
	and whence
	$\a = \frac{-19+\sqrt{1801}}{36} + o_M (1)$ as 
	%$\eps \to 0$ and 
	$M\to +\infty$.   
	Hence from \eqref{f:n_1}, say, we obtain $n\ge \frac{47+\sqrt{1801}}{3} + o_M (1) > 29.81 + o_M (1)$.
	Taking sufficiently large $M$,
	% and $\eps$ sufficiently small, 
	we can choose  $n=30$.  
	If $\a=0$, then for sufficiently large $M$ estimate \eqref{f:n_1} allows us to take $n = 23$. 
	This completes the proof. 
	$\hfill\Box$

\bigskip 
	
Combining the arguments 
%of this section 
above 
with Theorems \ref{t:8/3-c}, \ref{t:Misha_Je_large},  we obtain Theorem \ref{t:main_intr2} from the Introduction.
Actually, if we apply the second part of Theorem \ref{t:Misha_Je_large}, then we generate the whole $\SL_2 (\F_p)$ (and this differs our method from \cite{MOW_Schottky}, say). 
Because in the case $\zk =2$ we use results about growth in $\SL_2 (\F_p)$ for relatively small asymmetric set $A$ ($|A| \gg p^{2w_2} \gg p^{1.062}$)
our absolute constant $C$ is
% rather 
large. 
It is easy to see that the arguments of this section on trace of the set $A$ begin to work for $w_M > 3/4$ (see  Lemma \ref{l:A_cap_B},  
%and  
as well as 
estimates \eqref{tmp:15.10_1}, \eqref{tmp:15.10_1-}) and in this case the constant $C$ can be decreased, although it remains rather large.

%\section{Concluding remarks}

\end{document}